\documentclass[a4paper,12pt]{amsart}
\usepackage{amsfonts,amsthm,amsmath,amssymb,graphicx}

\title[Packing spectra for self-affine measures]{Packing Spectra for Bernoulli measures supported on Bedford-McMullen carpets}
\author{Thomas Jordan}
\address{Thomas Jordan, School of Mathematics, University of Bristol, University Walk, Bristol, BS8 1TW, UK}
\email{thomas.jordan@bris.ac.uk}
\author{Micha\l \ Rams} 
\address{Micha\l \ Rams, Institute of Mathematics, Polish Academy of Sciences\ ul. \'Sniadeckich 8, 00-956 Warszawa, Poland}
\email{rams@impan.pl}
\thanks{M.R. was supported by the MNiSW grant N201 607640 (Poland)}
\theoremstyle{plain}
\newtheorem{lem}{Lemma}[section]
\newtheorem{prop}[lem]{Proposition}
\newtheorem{thm}[lem]{Theorem}
\newtheorem{cor}[lem]{Corollary}

\theoremstyle{definition}
\newtheorem{defn}[lem]{Definition}

\theoremstyle{remark}

\numberwithin{equation}{section}

\newcommand{\R}{\mathbb R}

\newcommand{\N}{\mathbb N}

\newcommand{\dimp}{\dim_P}
\renewcommand{\epsilon}{\varepsilon}

\begin{document}

\begin{abstract}
In this paper we consider the packing spectra for local dimension of Bernoulli measures supported on Bedford-McMullen carpets. We show that typically the packing dimension of the regular set is smaller than the packing dimension of the attractor. We also consider a specific class of measures for which we are able to calculate the packing spectrum exactly and we show that the packing spectrum is discontinuous as a function on the space of Bernoulli measures. 
\end{abstract}
\maketitle
\section{Introduction}
The aim of this paper is to develop the theory of multifractal analysis for a special class of self-affine measures. These measures are the Bernoulli measures supported on Bedford-McMullen carpets, \cite{M}, \cite{B} and their multifractal properties have been studied in several papers: \cite{K}, \cite{O}, \cite{N}, \cite{BM}, \cite{GL}, \cite{JR}, \cite{R} and \cite{BF}. However most of these papers focus on the Hausdorff spectra and very little is known about the packing spectra. For self-similar measures satisfying the open set condition the packing and Hausdorff spectra are the same \cite{CM,AP} but are in general different for subsets of the irregular set where the liminf and limsup are specified, \cite{BOS}. For Bernoulli measures on Bedford-McMullen carpets in \cite{O} an upper bound is given in terms of the Legendre transform of a certain function, this is typically greater than the Hausdorff spectra. However in \cite{R} it is shown that this upper bound may not be sharp and that there are cases when the Hausdorff spectra and packing spectra are the same even when the Hausdorff and packing dimensions of the attractor are different. Moreover Reeve, \cite{R} calculated the packing spectra for a specific class of self-affine measures.

We extend this theory in two ways, firstly by showing that for a generic class of self-affine measures the packing dimension of the attractor is strictly greater than the maximum of the packing spectra, i.e. the packing dimension of the regular set. Secondly we consider a specific family of Bedford-McMullen carpets and Bernoulli measures supported on these carpets. We explicitly calculate the packing spectra for this family and show that it is not a continuous function of the parameters for the measures. This is in contrast to the case for the Hausdorff and packing spectra for self-similar measures, \cite{CM}, \cite{AP} and for the Hausdorff spectra of self-affine measures on Bedford-McMullen carpets, \cite{BM}, \cite{K}, \cite{O} and \cite{JR}.

The setting we consider is that for a given  Bernoulli measure $\mu$ supported on a Bedford-McMullen carpet, we look at the
set of points where the local dimension of $\mu$ equals $\alpha$
(to be denoted by $X_\alpha$) and the set of points where the
symbolic local dimension of $\mu$ (i.e. the local dimension
obtained by using approximate squares instead of geometric balls)
equals $\alpha$, to be denoted by $X_\alpha^{\rm symb}$. Then the
Hausdorff dimensions of $X_\alpha$ and  $X_\alpha^{\rm symb}$
coincide, they also coincide with the Hausdorff dimension of
certain Bernoulli measure $\mu_\alpha$, for which a typical point
belongs to $X_\alpha^{\rm symb}$. The function $\dim_H X_\alpha$
is concave, it is a Legendre transform of some well-defined
multifractal function. The maximum value achieved by $\dim_H
X_\alpha$ equals the Hausdorff dimension of the whole
Bedford-McMullen carpet. Moreover, $\dim_H X_\alpha$ is real
analytic both as a function of $\alpha$ and as a function of
$\mu$.

Seeing this, Olsen in \cite{O} conjectured that (most of) the same
properties hold for the packing spectrum. In particular, he
conjectured that $\dim_P X_\alpha^{\rm symb}$ is a Legendre
spectrum of another multifractal function, and then wrote down
some properties of such a function. However in \cite{R} it was
shown that for a certain family of examples the packing spectra is
the same as the Hausdorff spectra even when the packing dimension and Hausdorff dimension 
of the attractors is different.  In particular this shows that the
conjecture in \cite{O} cannot hold in general. This is related to
work by Nielsen, \cite{N} where the dimension of sets determined by the
frequencies of occurrence of each map in the iterated function
system. Again in this case Nielsen shows that the packing and
Hausdorff spectra are the same even when the Hausdorff and
packing dimension of the attractor are different.

In this note we go on to show that under a condition which holds
for typical carpets the conjecture in \cite{O} cannot hold and
then by considering a specific class of examples that the
situation is much more complicated than the situation with the
Hausdorff spectra. We prove two theorems. The first one, Theorem
\ref{thm:1}, states that it is unlikely that the maximum of
(symbolic) packing spectrum equals the packing dimension of the
Bedford-McMullen carpet. More precisely, there is a codimension
one condition necessary and codimension two condition sufficient
for this to happen. Both conditions are on coefficients of the
Bernoulli measure $\mu$.

Our Theorem \ref{thm:2} is more interesting. We consider a very
special family of Bernoulli measures, but the same phenomenon
holds in much greater generality. Assume the Bedford-McMullen
carpet has exactly two rows. We also assume that the Bernoulli measure $\mu$
is equally distributed in each row. For a given carpet such measures form a one
parameter family, they are uniquely determined by the measure of
the first row. For such measures we have $\dim_P X_\alpha=\dim_P
X_\alpha^{\rm symb}$ for all $\alpha$. Then, for all parameter
values except one, $\dim_H X_\alpha = \dim_P X_\alpha$ for all
$\alpha$. However, there is one exceptional measure $\mu$ in our
family for which $\dim_H X_\alpha < \dim_P X_\alpha$ for all
$\alpha$ in the interior of the spectrum interval. The function
$\dim_P X_\alpha$ is still well behaved as a function of
$\alpha$, but as a function of $\mu$ it is not even continuous.

Note here that this phenomenon is not an artifact created by the
fact that the packing dimension is in some way not adequate to
study the local dimension spectrum. On the contrary, while for all non-exceptional
measures $\mu$ the set $X_\alpha$ is equal to the set of
$\mu_\alpha$-typical points, for the exceptional measure it is
strictly greater. The Hausdorff dimension of those additional
points is equal to $\dim_H \mu_\alpha$ (which is why this set does not cause the Hausdorff
dimension to grow), but the packing dimension is
strictly greater than $\dim_H X_\alpha$ for all $\alpha$ in the
interior of the spectrum.

We are not able to present any conjectures, as to what the packing
spectrum really is in general. But it certainly seems to be an
object worthy of further study.

\section{Notation and results}

We start this section by defining packing dimension and stating some basic results we use to calculate the packing dimension of sets. For a set $A\subseteq\R^d$ and $s\geq 0$ let 
$$\tilde{\mathcal{P}}(A)=\lim_{\epsilon\to 0}\sup\left\{\sum_{i} B(x_i,r_i): x_i\in A,r_i<\epsilon\text{ and }|x_i-x_j|>r_i+r_j\text{ for }i\neq j\right\}$$
and define the outer packing measure by
$$\mathcal{P}(A)=\inf\left\{\sum \tilde{\mathcal{P}}(A_i):\cup_{i}A_i\supseteq A\right\}.$$
$\mathcal{P}$ gives a measure when restricted to measurable sets and we can define the packing dimension analogously to Hausdorff dimension by
$$\dim_P(A)=\inf\{s:\mathcal{P}(A)=0\}=\sup\{s:\mathcal{P}(A)=\infty\}.$$
While the main focus of this paper is on packing dimension we will at time use Hausdorff dimension, denoted by $\dim_H$, and upper box counting dimension denoted $\overline{\dim}_B$, for the definitions of these dimension we refer the reader to .
We will use the following standard results on packing dimension throughout the paper: the first result relates packing dimension to the Hausdorff dimension and the upper box dimension.
\begin{lem}
For any $A\subseteq\R^d$ we have that
$$\dim_H(A)\leq\dim_P(A)\leq\overline{\dim}_B(A).$$
\end{lem}
\begin{proof}
See   \cite{Mat}[page 82]. 
\end{proof}
The second is a version of Frostman's lemma for packing dimension.
\begin{lem}
For a Borel set $A\subseteq\R^d$ if there exists a Borel probability measure $\mu$ such that $\mu(A)=1$ and
$$\limsup_{r\to 0}\frac{\log\mu(B(x,r))}{\log r}\geq s$$
for all $x\in A$ then $\dim_P A\geq s$.
\end{lem}
\begin{proof}
See \cite{Mat}[Theorem 6.11, page 97].
\end{proof}

 We now formally introduce Bedford-McMullen carpets. Let $m,n\in\N$ with $2\leq m<n$ and 
$$D\subseteq \{0,\ldots,m-1\}\times\{0,\ldots,n-1\}$$
where $|D|\geq 2$ and we define $\sigma:=\frac{\log m}{\log n}.$ 
For $(i,j)\in D$ let $T_{i,j}:\R^2\to\R^2$ be defined by 
$$T_{i,j}((x,y))=\left(\frac{x+i}{n},\frac{y+j}{m}\right)$$
and $\Lambda$ be the unique non-empty compact set satisfying
$$\Lambda=\cup_{(i,j)\in D}T_{{i,j}}(\Lambda).$$
Sets of the form $\Lambda$ were first studied in \cite{B} and \cite{M} and are usually known as Bedford-McMullen carpets. We will let $L_0=|D|$, and for $0\leq i\leq m-1$ we let
$$n_i=|\{(i,j)\in D:0\leq j\leq n-1\}|.$$
Finally let $L_1=|\{i:n_i\neq 0\}|$.

The geometry of affine sets is different than the geometry of
conformal fractals because the cylinders are not approximate balls
anymore. Hence, in order to work with the geometric properties of
the fractal, we need to use not single cylinders but some unions
of cylinders. Given two positive integers $n_1\leq n_2$ and a
point $x\in\Lambda$ we define the {\bf rectangle}
\[
R^{n_1, n_2}(x) = \{y\in \Lambda; i_k(y)=i_k(x) \ \forall k\leq
n_2 \rm{ and } j_k(y)=j_k(x) \ \forall k\leq n_1\}.
\]
Note that $R^{N,N}(x)$ is just the $N$-th level cylinder
containing $x$. The rectangle $R^{n_1, n_2}(x)$ is an intersection
of $\Lambda$ with a geometrical rectangle with horizontal side
$n^{-n_1}$ and vertical side $m^{-n_2}$. The $N$-th level {\bf
approximate square} is defined as $C_N(x) = R^{\lceil\sigma
N\rceil, N}(x)$. It is a set of diameter approximately $m^{-N}$. It will sometimes be convenient to define $C_N(x)$ for non-integer values of $N>0$, we will denote $C_{N}(x)=C_{\lceil N\rceil}(x)$.

We now introduce the class of measures we will be considering which are projections of Bernoulli measures to Bedford-McMullen carpets. We let $\{p_{i,j}\}_{(i,j)\in D}$ be a probability vector and $\mu$ be the unique Borel probability measure for which for all Borel sets $B\subseteq\R^2$ we have
$$\mu(B)=\sum_{(i,j)\in D} p_{i,j}\mu(T_{i,j}^{-1}(B)).$$
An alternative way of defining this measure is as the natural projection of $\{p_{i,j}\}$ Bernoulli measure from the shift space $D^{\N}$ to $\Lambda$.

In this paper we deal with local dimensions.
The {\bf real local dimension} (or simply local dimension; we use
the name real local dimension to distinguish it from the symbolic
local dimension) of a Borel measure $\mu$ supported in $\R^2$ at a point $x$ is defined as
\[
d_\mu(x) = \lim_{r\to 0} \frac {\log \mu(B_r(x))} {\log r},
\]
provided the limit exists. 
However it is often easier to work with approximate squares rather then geometric balls.
Thus we define the {\bf symbolic local dimension} as
\[
\delta_\mu(x) = \lim_{N\to\infty} \frac {\log \mu(C_N(x))} {-N\log
m},
\]
provided the limit exists. For any $\alpha\in\R$ we will denote
$$X_{\alpha}=\{x\in\Lambda:d_\mu(x)=\alpha\}\text{ and }X_{\alpha}^{\rm{symb}}\{x\in\Lambda:\delta_\mu(x)=\alpha\}.$$
In several cases it is possible to show that $\delta_\mu(x)=d_{\mu}(x)$ however this is not always true. We will let
$$\alpha_{m}=\min_{(i,j)\in D}\frac{ -\sigma\log p_{ij}+(\sigma-1)\log q_i}{\log m}$$
and
$$\alpha_{M}=\max_{(i,j)\in D}\frac{ -\sigma\log p_{ij}+(\sigma-1)\log q_i}{\log m}.$$
Note that $\alpha\notin [\alpha_m,\alpha_M]$ is equivalent to $X_{\alpha}=X_{\alpha}^{\rm{symb}}=\emptyset$.

Clearly, the relation between symbolic and real local dimension at
any given point is given by the geometric interplay between balls
and approximate squares. On the one hand, we have for all $x\in
\Lambda$:
\begin{equation} \label{above}
C_N(x) \subseteq B_{m^{-N}}(x).
\end{equation}

On the other hand, the point $x$ can be very close to the boundary
of an approximate square it lies in, and then the ball
$B_{cm^{-N}}(x)$ will only be contained in $C_N(x)$ for very small
$c$. We can describe the distance from $x$ to the boundary of
$C_N(x)$ using the symbolic expansion of $x$. Denote

\[
I_N(x) = \frac 1 N \sup\{k\in \N; i_{N+1}(x)=\ldots=i_{N+k}(x)\in
\{0, m-1\}\}
\]
and

\[
J_N(x) = \frac 1 {\lceil \sigma N\rceil} \sup\{k\in \N; j_{\lceil
\sigma
  N\rceil+1}(x)=\ldots=j_{\lceil \sigma N\rceil+k}(x)\in\{0, n-1\} \}.
\]

If some positive $\delta$ is greater than $I_N(x)$ and  $J_N(x)$
then the distance from $x$ to the boundary of $C_N(x)$ is at least
$m^{-N(1+\delta)}$, and hence

\begin{equation} \label{below}
B_{m^{-N(1+\delta)}}(x) \subseteq C_N(x).
\end{equation}

An easy consequence of \eqref{above}, \eqref{below} is that if
both $I_N(x)$ and $J_N(x)$ converge to 0 as $N$ goes to infinity
(in particular, almost all points for any Bernoulli measure with
nontrivial horizontal and vertical projections have this property)
then the symbolic and local dimensions of any Bernoulli measure
coincide at $x$.

The symbolic local dimension can thus be used to calculate the
real local dimension at many points. At the same time, for any
Bernoulli measure $\mu$ its symbolic local dimension at a point
$x$ is easy to calculate from symbolic expansion of $x$. Denote

\[
q_i = \sum_j p_{i,j}.
\]

We then have

\begin{equation} \label{mucn}
\mu(C_N(x)) = \prod_{k=1}^{\lceil \sigma N\rceil} p_{i_k(x),
j_k(x)} \cdot \prod_{k=\lceil \sigma N\rceil +1}^N q_{i_k(x)}
\end{equation}

and we can calculate $\delta_\mu(x)$ directly. Note also another
important consequence of \eqref{mucn}: there exist a constant
$K>0$ such that for any $x\in\Lambda$ and $N>0$ we have

\begin{equation} \label{smoothmucn}
K < \frac {\mu(C_{N+1}(x))} {\mu(C_N(x))} \leq 1.
\end{equation}

Our results are as follows.

\begin{thm} \label{thm:1}
The symbolic and real packing local dimension spectra are both strictly smaller
than $\dim_P \Lambda$ for systems not satisfying

\begin{equation} \label{condition}
\sum_i \frac 1 {L_1} \log q_i = \sum_i \frac {n_i} {L_0} \log q_i = A.
\end{equation}

For systems satisfying both \eqref{condition} and

\begin{equation} \label{condition2}
\sum_{(i,j)\in D} \frac {1}{L_0} \log p_{ij} = \sum_{(i,j)\in D} \frac
1 {n_i L_1} \log p_{ij}=B
\end{equation}
both the real and symbolic packing spectra are equal to $\dim_P \Lambda$ at
\[
\alpha_0= -\frac 1 {\log m} (\sigma B + (1-\sigma) A).
\]
\end{thm}

Consider now a special class of systems. Assume that the
Bedford-McMullen carpet has only two rows with $n_0$ and $n_1$
rectangles in each. Let $\mathcal{M}$ be the class of Bernoulli
measures with probabilities equidistributed in each row. Denote by
$p_0$ the probability of each rectangle in the first row and by
$p_1$ the probability of each rectangle in the second row; the
condition $n_0 p_0 + n_1 p_1 = 1$ must hold with $q_0=n_0p_0$ and
$q_1=n_1p_1$.

\begin{thm} \label{thm:2}
For $\mu\in \mathcal{M}$, we have that for $\alpha\in
[\alpha_m,\alpha_M]$
\begin{enumerate}
\item\label{thm2p2} If $\mu\in \mathcal{M}$ and
$\frac{\log(q_0/q_1)}{\sigma\log(n_0/n_1)}\neq -1$ then
\[
\dim_P X_\alpha = \dim_P X_{\alpha}^{\rm{symb}}=\dim_H X_{\alpha}
\]
\item\label{thm2p3} If $\mu\in\mathcal{M}$ is the unique measure
for which $\frac{\log(q_0/q_1)}{\sigma\log(n_0/n_1)}=-1$ then for
$\alpha\in (\alpha_m,\alpha_M)$ we have that $$ \dim_H
X_{\alpha}<\dim_P X_{\alpha}=\dim_P X_{\alpha}^{\rm{symb}}$$ and for
$\alpha\in\{\alpha_m,\alpha_M\}$
$$ \dim_H
X_{\alpha}=\dim_P X_{\alpha}=\dim_P X_{\alpha}^{\rm{symb}}.$$
\end{enumerate}
\end{thm}

\section{Proof of Theorem \ref{thm:1}}
The main part of the proof is to show that \eqref{condition} is necessary for the real and symbolic packing spectra to achieve $\dim_{P}(\Lambda)$. To prove this for a fixed Bernoulli measure we construct two sets of dimension strictly smaller than $\dim_{P}(\Lambda)$ and then prove that the first set contains all symbolically regular points and that the second set contains all regular points. The second part of the Theorem, that satisfying both conditions \eqref{condition} and \eqref{condition2} is sufficient for the maximum of the real and symbolic packing spectra to achieve $\dim_{P}(\Lambda)$, is easy to show.

Let $\mu$ be the Bernoulli measure defined by the probability
vector $\{p_{ij}\}$ defined on the digit set $D\subseteq
\{0,\ldots,m-1\} \times \{0,\ldots,n-1\}$, $m<n$. We assume that
$D$ contains at least two different $i$'s and at least two
different $j$'s, otherwise the system would be a self-similar IFS
on the line. We have

\[
\dim_B\Lambda = \dim_P\Lambda = s :=\frac 1 {\log m} (\sigma \log
L_0 + (1-\sigma) \log L_1),
\]
where $\sigma=\log m/\log n$.

Let us start with a simple geometric lemma. For $k_1,k_2\in\N$ with $K_1< k_2$ we denote by $F_{k_1}^{k_2}(i,j)(x)$
 the frequency of symbol $(a,b)$ in the
sequence $(i,j)_{k_1+1}(x),\ldots,(i,j)_{k_2}(x)$. More precisely
$$F_{k_1}^{k_2}(i,j)(x):=\frac{\#\{(i_l(x),j_l(x))=(i,j):k_2<l\leq k_1\}}{k_2-k_1}.$$
 We will extend this notation to the case when $k_1$ and $k_2$ are not integers and $k_2-k_1\geq 1$ in which case $F_{k_1}^{k_2}(i,j)(x)$  will denote the frequency of the symbol $(i,j)$ in the sequence $(i,j)_{\lceil k_1\rceil+1}(x),\ldots,(i,j)_{\lceil k_2\rceil}(x)$.
For $a>0$ we let $Z(N,a)$ be the set of points $x\in\Lambda$ such
that for any $M>N$ one of the following seven {\bf nongenericity}
conditions holds:

\begin{itemize}

\item[i)] $F_0^M(i,j)(x) \notin [1/L_0-a, 1/L_0+a]$ for some $(i,j)\in
D$, \item[ii)] $F_M^{\lceil M \sigma^{-1} \rceil}(i,j)(x) \notin
[1/L_0-a, 1/L_0+a]$ for some $(i,j)\in D$, \item[iii)] $\sum_j
F_{\lceil M \sigma^{-1}\rceil}^{\lceil M
    \sigma^{-2}\rceil}(x) \notin [1/n_i-a,1/n_iL_1+a]$ for some $i$,
\item[iv)] all the symbols $j_{M+1}(x),\ldots,j_{\lceil M (1+a)
\rceil }(x)$ are equal, \item[v)] all the symbols $i_{\lceil M
\sigma^{-1} \rceil
    +1}(x),\ldots,i_{\lceil M \sigma^{-1}(1+a) \rceil }(x)$ are equal.
\item[vi)] all the symbols $j_{\lceil M \sigma^{-1} \rceil
    +1}(x),\ldots,j_{\lceil M \sigma^{-1}(1+a) \rceil }(x)$ are equal,
\item[vii)] all the symbols $i_{\lceil M \sigma^{-2} \rceil
    +1}(x),\ldots,i_{\lceil M \sigma^{-2}(1+a) \rceil }(x)$ are equal.
\end{itemize}

Let $\widetilde{Z}(N,a)$ be the subset of $Z(N,a)$ consisting of
points $x\in\Lambda$ such that for any $M>N$ one of the first
three nongenericity conditions i), ii) or iii) holds.

\begin{lem} \label{box}
For any $a>0$

\[
\sup_N \overline{\dim}_B Z(N,a) < s.
\]
\end{lem}

\begin{proof}
Let us begin by calculating the upper box counting dimension of
$\widetilde{Z}(N,a)$. The general idea of the proof is that we have
approximately $z_M= L_0^{M \sigma^{-1}} L_1^{M
  (\sigma^{-2} - \sigma^{-1})}$ approximate squares of level $\lceil M
\sigma^{-2}\rceil$ but for $N<M$ $\widetilde{Z}(N,a)$ intersects
at most $z_M e^{-c M a^2}$ of them.

There are three types of points $x\in \widetilde{Z}(N,a)$. If  $F_0^M(i,j)(x) \notin [1/4L_0-a, 1/L_0+a]$ for some $(i,j)\in D$
then there are only $L_0^{M(1-c_1a^2)}$ possible values of the
initial $M$ symbols $(i,j)_k(x)$. Indeed, the entropy of Bernoulli
measure with probabilities $p_1,\ldots,p_{L_0}$ is $\log L_0$, achieved
when all those probabilities are equal to $1/L_0$, and this maximum
is nonflat. Hence, all points of this type can be covered by at
most $L_0^{M(1-c_1a^2)} L_0^{\lceil M\sigma^{-1}\rceil -M} L_1^{\lceil M
  \sigma^{-2}\rceil - \lceil M \sigma^{-1} \rceil} \approx z_M L_0^{-c_1 M a^2}$ approximate squares of level $\lceil
M\sigma^{-2} \rceil$.

If  $F_M^{\lceil M\sigma^{-1} \rceil}(i,j)(x) \notin [1/L_0-a,
1/L_0+a]$ for some $(i,j)\in D$ then there are only $L_0^{(\lceil
M\sigma^{-1} \rceil -M)(1-c_2 a^2)}$ possible values of the
symbols $(i,j)_k(x)$ for $k=M+1,\ldots,\lceil M\sigma^{-1}\rceil$.
Hence, by a reasoning similar to the previous one, all points of
this type can be covered by at most $z_M L_0^{-c_2 a^2 M
(\sigma^{-1} -1)} $ approximate squares of level $\lceil
M\sigma^{-2} \rceil$.

If  $\sum_j F_{\lceil M \sigma^{-1}}^{\lceil M
    \sigma^{-2}}(i,j)(x) \notin [1/n_iL_1 - a, 1/n_iL_1+a]$ for some $i$ then there
  are only $l_1^{(\lceil M
  \sigma^{-2}\rceil - \lceil M \sigma^{-1} \rceil) (1-c_3 a^2)}$ possible
values of the symbols $i_k$ for $k=\lceil
M\sigma^{-1}\rceil+1,\ldots,\lceil M\sigma^{-2}\rceil$.  Hence,
all points of this type can be covered by at most $z_M   L_1^{-c_3
a^2 M(
  \sigma^{-2} - \sigma^{-1})}$  approximate squares of level $\lceil
M\sigma^{-2} \rceil$.

The points in $\widetilde{Z}(N,a)$ might satisfy different
nongenericity conditions for different $M$. However, for each
$M>N$ all the points in $\widetilde{Z}(N,a)$ can be covered by at
most $z_M e^{-c M a^2}$ approximate squares of level $\lceil M
\sigma^{-2} \rceil$. Hence,

\[
\overline{\dim}_B \widetilde{Z}(N,a) \leq s - \tilde{c} a^2.
\]

Consider now the set $Z(N,a)$. Instead of using approximate
squares of level $\lceil M\sigma^{-2} \rceil$, we will use squares
of level $\lceil M\sigma^{-2} (1+a/4) \rceil$. Note first, that if
$x\in \Lambda$ satisfies one of nongenericity conditions i), ii),
iii) for $M$ and $a$, it will also satisfy them for $\lceil
M(1+a/4) \rceil$ and $a/2$.

Indeed, if $x$ satisfies nongenericity condition i) then the symbols
$((i,j)_k(x))_{k=0}^{\lceil M(1+a/4) \rceil}
$ are the same symbols as
$((i,j)_k(x))_{k=0}^{M}$ (where frequencies differ from
$(1/L_0,\ldots,1/L_0)$ by at least $a$) plus approximately $Ma/4$ new
symbols $((i,j)_k(x))_{k=M+1}^{\lceil M(1+a/4) \rceil}$ which can
change the frequencies at most by $a/4$. If $x$ satisfies either
nongenericity ii) or iii) then passing from $M$ to $\lceil
M(1+a/4) \rceil$ changes the considered ranges of $k$ on both
ends: some symbols in the beginning drop out, some symbols at the
end are added. Altogether the change of frequencies cannot top $a
(2+\sigma)/4(1+\sigma)<a/2$.

Hence, all the points $x\in\Lambda$ satisfying nongenericity conditions i),
ii) or iii) for given $M$ and $a$ can be covered with at most
$z_{M(1+a/4)} m^{-\tilde{c}
  M \sigma^{-2}(1+a/4) a^2/4}$ sets of
diameter $m^{-\lceil M
  \sigma^{-2}(1+a/4)\rceil}$, like in the first part of proof.

If $x\in\Lambda$ satisfies nongenericity condition iv) then the sequence
$((i,j)_k(x))_{k=M+1}^{\lceil M (1+a/4) \rceil}$ can take only
$(L_0-1)^{ \lceil a M \rceil/4}$ possible values. Hence, the points
of this type can be covered by at most  $z_{M(1+a/4)} ((L_0-1)/L_0)^{
  M a/4}$ approximate squares of level $\lceil
M\sigma^{-2} (1+a/4) \rceil$.

The cases of nongenericity conditions v) or vi) are almost identical, the
sequence $((i,j)_k(x))_{\lceil M \sigma^{-1}+1 \rceil}^{\lceil M
\sigma^{-1} (1+a/4) \rceil}$ can take only $(L_0-1)^{ \lceil a
\sigma^{-1} M/4 \rceil}$ possible values and we can cover those
points with at most  $z_{M(1+a/4)} ((L_0-1)/L_0)^{
  M \sigma^{-1} a/4}$ approximate squares of level $\lceil
M\sigma^{-2} (1+a/4) \rceil$.

Finally, if  $x\in\Lambda$ satisfies nongenericity condition vii) then the
sequence $\{i_k(x)\}_{k=\lceil M \sigma^{-2} \rceil+1}^{\lceil M
\sigma^{-2} (1+a/4) \rceil}$ can take only $(L_1-1)^{ \lceil a M
\sigma^{-2}/4 \rceil}$ possible values. Hence, the points of this
type can be covered by at most  $z_{M(1+a/4)}((L_1-1)/L_1)^{
  M \sigma^{-2} a/4}$ approximate squares of level $\lceil
M\sigma^{-2} (1+a/4) \rceil$.

As $a$ is small, $a^2$ is small in comparison to $a$. Hence, again we get a quadratic bound
(independent from $N$) for the upper box counting dimension of
$Z(N,a)$:

\[
\overline{\dim}_B Z(N,a) \leq s - c a^2.
\]
\end{proof}

We now proceed with the proof of the theorem. Let us consider the
symbolic local dimensions first. Assume that \eqref{condition}
does not hold and let

\[
\delta = \left| \frac {1-\sigma} {3\log m} \cdot \sum_i
\left(\frac 1 {L_1} -\frac
  {n_i}{L_0}\right) \log q_i \right|.
\]

Let
\[
X_{\alpha, N, \delta} = \{x\in\Lambda; \alpha - \delta < \frac
{\log \mu(C_{N_0}(x))} {-N_0\log m} < \alpha + \delta \ \forall
N_0>N\}.
\]

We have

\[
X_{\rm regsymb} \subseteq \bigcup_\alpha \bigcup_N X_{\alpha, N,
\delta}.
\]

Let $x\in X_{\alpha, N,\delta}$. Let $a_0\geq 0$ be the smallest
number for which the following are true:

\begin{itemize}
\item[i)] $1/L_0-a_0 \leq F_0^M(i,j)(x) \leq 1/L_0+a_0$ for all
$(i,j)\in D$, 
\item[ii)] $1/L_0-a_0 \leq F_M^{ M \sigma^{-1}
}(i,j)(x) \leq  1/L_0+a_0$ for all $(i,j)\in D$, 
\item[iii)]
$1/n_iL_1 -a_0 \leq \sum_j F_{ M \sigma^{-1}}^{ M
    \sigma^{-2}}(i,j)(x) \leq 1/n_iL_1+a_0$ for all $i$.
\end{itemize}

By \eqref{mucn}, we have
\begin{equation} \label{e1}
- \frac {\log m} {M\sigma^{-1}}\log \mu(C_{
M\sigma^{-1}}(x)) = \sigma \sum_{(i,j)\in D} \frac 1 {L_0} \log
p_{i,j}
  +(1-\sigma) \sum_i \frac {n_i} {L_0} \log q_i +Z_1
\end{equation}
and
\begin{equation} \label{e2}
- \frac {\log m} {M\sigma^{-2}}\log \mu(C_{
M\sigma^{-2}}(x)) = \sigma  \sum_{(i,j)\in D} \frac 1 {L_0} \log
p_{i,j} + (1-\sigma) \sum_i \frac 1{L_1} \log q_i + Z_2,
\end{equation}

where

\[
|Z_1|, |Z_2| < a_0 \cdot \max(\max_{i,j} |\log p_{i,j}|, n \max_i
|\log q_i|) + O(1/M).
\]

Denoting $T= \max(\max_{i,j} |\log p_{i,j}|, n \max_i |\log
q_i|)$, we see that, as the left hand sides of \eqref{e1} and
\eqref{e2} can differ at most by $2\delta \log m$ and the right
hand sides differ at least by $3\delta \log m - |Z_1| - |Z_2|$, we
must have

\[
a_0 > a= \frac {\log m} {2T} \cdot \delta.
\]

Hence,
\[
X_{\rm regsymb} \subseteq \bigcup_N \widetilde{Z}(N,a)
\]
and the symbolic part of the assertion follows by Lemma \ref{box}.

Now consider the regular points for real local dimension. Let
$x\in X_{\rm
  reg}$. There exist $\alpha$ (which can be chosen from a finite set) and
$N>0$ such that for all $M>N$

\[
\frac {\log \mu(B_{m^{-M}}(x))} {-M \log m} \in [\alpha -
\delta/2, \alpha + \delta/2].
\]

There are two cases: either the inequality

\[
\frac {\log \mu(C_{\tilde{M}}(x))} {-\tilde{M} \log m} \in [\alpha
- \delta, \alpha + \delta]
\]
holds for both $\tilde{M}=\lceil M \sigma^{-1} \rceil$ and
$\tilde{M}=\lceil M \sigma^{-2} \rceil$ or it does not hold for at
least one of those. If it holds then, like above, $x$ must satisfy
 nongenericity condition i), ii) or iii) for $M$ and $a$. If it does not hold
for one of possible values of $\tilde{M}$ then, necessarily, the
measures of $C_{\tilde{M}}(x)$ and $B_{m^{-\tilde{M}}}(x)$ must
differ at least by factor $m^{\tilde{M}\delta/2}$. However, by
\eqref{above}, \eqref{below} and
  \eqref{smoothmucn} it is only possible if either $I_{\tilde{M}}(x)$ or $J_{\tilde{M}}(x)$ are
  greater than $\tilde{a}=\delta |\log K|/2\log m$. Hence, in this case $x$
  must satisfy nongenericity condition iv), v), vi) or vii) for $M$ and $\tilde{a}$. We came to
  conclusion that

\[
X_{\rm reg} \subseteq \bigcup_N Z(N, \min(a, \tilde{a}))
\]
and we are done by Lemma \ref{box}.

For the second statement of the theorem, if both \eqref{condition}
and \eqref{condition2} hold, all points $x\in \Lambda$ which
symbolic expansions can be divided into parts in which symbols
$(i,j)\in D$ appear with frequencies $\{1/L_0,\ldots,1/L_0\}$ and
parts with frequencies $\{1/n_iL_1\}$ belong to $X_{\alpha_0}^{\rm
symb}$. If in addition we demand that they do not have long
stretches of identical $i$'s or $j$'s then the real local dimension
at those points is also $\alpha_0$. It is easy to check that those
points have full packing dimension.

\section{Proof of Theorem \ref{thm:2}}
The proof is divided into four parts. In the first part we relate the symbolic and real local dimensions at any point. We continue by looking at how the local dimensions can be determined by observing the frequency of digits in initial parts of symbolic expansions. This now naturally splits the argument into two cases which are the final two parts of the proof. The first case corresponds to part \ref{thm2p2} of the Theorem and the second case corresponds to part \ref{thm2p3} of the Theorem. In fact part \ref{thm2p3} of the Theorem is about only one measure but it will turn out this is the only case when the Hausdorff and packing spectra are different and it is by far the most difficult case.

\subsection{Real and symbolic local dimensions}
We start by studying the relationship between the symbolic and real local dimensions. 
\begin{lem}\label{boundary}
If $\alpha\in [\alpha_{\min},\alpha_{\max}]$ and $d_{\mu}(x)=\alpha$ but $d_{\mu}^{\rm{symb}}(x)\neq \alpha$ then the symbolic expansion, $(i_k,j_k)$ for $x$ satisfies that there must exist $\eta>0$ and infinitely many integers $n_j$
such that $i_{n_j}=i_{n_j+k}$ for all $0\leq k\leq [\eta n_j]$.
\end{lem}
\begin{proof}
To start the proof fix a positive integer $N$ and let $k(N)=\inf\{z:i_{N}\neq i_{N+z}\}$. Recall that the ball $B_{cm^{-N}}(x)$
contains $C_N(x)$ for $c=n+1$ and if $c=m^{-(k+1)}$, it is contained in $C_N(x)\cup C_N(x')$ for some $x'$ for which $i_k(x')=i_k(x) \forall k$. Note that either $\mu(C_N(x'))=0$ or $\mu(C_N(x'))=\mu(C_N(x))$. Hence
$$\mu(B_{(n+1)m^{-N}}(x)\leq\mu(C_N(x))\leq 2\mu(B_{(m^{-k+1}m^{-N}}(x))$$
and if the symbolic and real local dimensions are not equal then we cannot have that $K(N)=\mathit{o}(N)$ and the result follows.
\end{proof}

\begin{lem} \label{lem:eqsymb}
\[
X_{\alpha} \subseteq X_{\alpha}^{\rm{symb}}
\]
\end{lem}
\begin{proof}

Let $x\in X_{\alpha}$. For any $\epsilon>0$ there must exist $N_0$ such that
$$x\in\{x\in\Lambda:\mu(B(x,m^{-N}))\in (m^{-N(\alpha-\epsilon)},m^{-N(\alpha-\epsilon)})\text{ for all }N\geq N_0\}.$$
We will denote this set by $Y_{\alpha,N_0,\epsilon}$. By Lemma \ref{boundary} there must exist $\delta>0$ and infinitely
many $N\geq N_0$ such that $K(N)/N\geq\delta$. We will fix such an $N>\text{something}$ and assume without loss of generality that $i_N=1$. For $c>n$
and $M\in [N,N(1+\delta)]$ the ball $B_{cm^{-M}}(x)$ will contain both
the approximate square $C_M(x)$ and approximate square $C_M(y_N)$,
where $y_N$ is the point with the same symbolic expansion as $x$
except $i_k(y)=1-i_k(x)$ for $k\in [N,N(1+\delta)]$. Similarly,
for $c$ small the ball $B_{cm^{-M}}(x)$ will be contained in
$C_M(x)\cup C_M(y_N)\cup C_M(x')\cup C_M(y_N')$. Hence, as $x\in Y_{\alpha,N_0,\epsilon}$
 for all $M\in [N,N(1+\delta)]$ we must have
\[
1/4m^{-M(\alpha+2\epsilon)}\leq\max(\mu(C_M(x)), \mu(C_M(y_N)))\leq m^{-M(\alpha-2\epsilon)}.
\]
We can also assume that for at least some $N$ and $M$ this
maximum is not $\mu(C_M(x))$, otherwise the symbolic and real
local dimensions at $x$ would be equal. There are now several
cases to consider:

\textbf{Case I}: $q_0=q_1$, $\delta\leq\sigma^{-1}$. In this situation
$\mu(C_M(x))= \mu(C_M(y_N))$ for all $M\in [N,N(1+\delta)]$, because
the measure of an approximate square of level $M$ only depends on
the initial $M\sigma$ symbols from the symbolic expansion, which
are the same for $x$ and for $y_N$.

\textbf{Case II}: $q_0=q_1$, $\delta>\sigma^{-1}$. In this situation we
have $\mu(C_M(x))= \mu(C_M(y_N))$ for all $M\in [N,N\sigma^{-1}]$
like in case I, but for $M\in [N\sigma^{-1}, N(1+\delta)]$ we have

\[
\frac{\mu(C_M(y_N))}{\mu(C_M(x))} = \left( \frac{p_1}{p_0}
\right)^{\sigma M-N-1}.
\]

As $\mu(C_M(y_N))$ is greater for at least some $M$, we must have
$p_1>p_0$. Hence, $-1/M \log(\mu(C_M(y_N))/\mu(C_M(x)))$ is a
monotonically increasing function. By our assumptions about $x$ we may assume that for all $M\in [N,N(1+\delta)]$ ,
\[
\mu(C_M(y_N))\in (1/4m^{-M(\alpha+2\epsilon)},m^{-M(\alpha-2\epsilon)})
\]
and there exists $M\in [N,N(1+\delta)]$ such that
\[
\mu(C_M(x))<1/4m^{-M(\alpha+2\epsilon)}.
\]
Since $-1/M \log(\mu(C_M(y_N))/\mu(C_M(x)))$ is monotonically increasing in $M$ we have that
\[
\mu(C_{N(1+\delta)}(x)) < m^{-N(1+\delta)(\alpha+2\epsilon)}
\]
At the same time, our assumptions give us
\[
\mu(C_{N(1+\delta)}(x)) \geq m^{-N(1+\delta)(\alpha+2\epsilon)}
\]
which is a contradiction.

\textbf{Case III}: $q_0\neq q_1$, $\delta\leq\sigma^{-1}$. We have for all
$M\in [N,N(1+\delta)]$

\[
\frac {\mu(C_M(y_N))} {\mu(C_M(x))} = \left( \frac {q_1} {q_0}
\right)^{M-N-1}.
\]

Like in case II, $-1/M \log(\mu(C_M(y_N))/\mu(C_M(x)))$ is a
monotonically increasing function, we must have for all $M\in
[N,N(1+\delta)]$

\[
\mu(C_M(y_N)) \in (1/4m^{-M(\alpha+2\epsilon)},m^{-M(\alpha-2\epsilon)})
\]
but we also have

\[
\mu(C_{N(1+\delta)}(x)) \approx m^{-N(1+\delta)\alpha},
\]
which leads to a similar contradiction.

\textbf{Case IV}: $q_0\neq q_1$, $\delta>\sigma^{-1}$. This time, $-1/M
\log(\mu(C_M(y_N))/\mu(C_M(x)))$ is a monotonically increasing
function for $M\in [N,N\sigma^{-1}]$ but might be monotonically
decreasing for $M\in [N\sigma^{-1}, N(1+\delta)]$. Hence, if $\delta$
satisfies

\begin{equation} \label{cond3}
m^{-2N(1+\delta)\epsilon}\leq\left( \frac {q_1} {q_0} \right)^{N(1+\delta) (1-\sigma)} \left( \frac
{p_1} {p_0} \right)^{N((\delta+1)\sigma-1)}\leq m^{2N(1+\delta)\epsilon}
\end{equation}
then it is possible that
\[
\frac 1 {N(1+\delta)\log m} \left|\log \mu(C_{N(1+\delta)}(x)) -
\log \mu(C_{N(1+\delta)}(y_N)) \right| < 2\epsilon
\]
and the contradiction like in cases II, III
does not happen. However, note that

\[
\mu(C_{N(1+\delta)}(y_N)) = \mu(C_{N\sigma^{-1}}(y_N)) \cdot
q_1^{N(1+\delta - \sigma^{-1})} p_1^{N((1+\delta)\sigma-1)}.
\]
As

\[
\frac {1}{N(1+\delta)} \log \mu(C_{N(1+\delta)}(y_N))\in ((-\alpha-2\epsilon)\log m,(-\alpha+2\epsilon)\log m)\] 
and 
\[\frac {1}{N\sigma^{-1}} \log \mu(C_{N\sigma^{-1}}(y_N))\in ((-\alpha-2\epsilon)\log m,(-\alpha+2\epsilon)\log m)
\]
 we must have that
\[
-\alpha \log m = \sigma \log p_1 + (1-\sigma) \log q_1,
\]
which implies $\alpha=\alpha_{\text{min}}$.

Consider now $M=N(1+\delta)\sigma^{-1}$. In the symbolic expansion of
$x$ there are at least $N\delta$ zeros on the first
$N(1+\delta)$ places. Hence, the measure of any approximate square of
level $N(1+\delta)\sigma^{-1}$ that can intersect the ball
$B_{m^{-N\delta\sigma^{-1}}}(x)$ (the symbolic descriptions of
those squares share the first $N(1+\delta)$ symbols with $x$) is at most

\[
p_1^N p_0^{N\delta} q_0^{N(1+\delta)(\sigma^{-1}-1)} = m^{-\alpha
N(1+\delta)\sigma^{-1}} \cdot \left(\frac {p_0} {p_1} \right)^{N\delta} \cdot \left(\frac {q_0} {q_1}\right)^{N(1+\delta)(\sigma^{-1} -1)}
\]
(remember that $q_1<q_0$ and $p_1>p_0$ by \eqref{cond3}). By \eqref{cond3},
\[
\left(\frac {p_0} {p_1} \right)^{N\delta} \cdot \left(\frac {q_0} {q_1}\right)^{N(1+\delta)(\sigma^{-1} -1)} \leq m^{2N(1+\delta)\sigma^{-1}\epsilon} \cdot \left(\frac {p_0} {p_1}\right)^{N(\sigma^{-1}-1)} < m^{-N(1+\delta)\sigma^{-1} \epsilon}.
\]

Hence,
$x$ cannot belong to $Y_{\alpha, N_0, \epsilon}$.
\end{proof}

\subsection{Local dimensions and frequencies}
We now look at the relationship between frequencies of digits and local dimensions.
To do this we need to introduce some notation adapted to this setting. Recall that in this case we have two rows and the measure $\mu$ is a Bernoulli measure giving weight $q_0/n_0$ to each rectangle in the first row and $q_1/n_1$ to each rectangle in the second row. Using this for $P\in (0,1)$ we will define the following quantities:
\begin{enumerate}
\item
$H(P)=-P\log P-(1-P)\log(1-P)$,
\item
$H_q(P)=-P\log q_0-(1-P)\log q_1$,
\item
$CH(P)=P\log (n_0^{-1}P)-(1-P)\log (n_1^{-1}P)$,
\item
$CH_q(P)=P\log (n_0^{-1}q_0)-P\log (n_1^{-1}q_1)$.
\end{enumerate}
It will also be convenient to let $G_k=\sum_{j=0}^{n-1} F_0^{k-1}(0,j)(x)$

Let $\mu\in\mathcal{M}$, $\alpha\in [\alpha_m,\alpha_M]$ and $x\in X_{\alpha}^{\rm{symb}}$.
For any positive $\epsilon$ we have 
\begin{equation} \label{epsdef}
X_\alpha^{\rm symb} \subseteq	 \bigcup_{N_0} \bigcap_{N>N_0} X_{\alpha, N, \epsilon}.
\end{equation}
We fix some small positive $\epsilon$ (in the future we will
determine, how small it should be). We assume that $x\in \bigcap_{N>N_0} X_{\alpha, N, \epsilon}$ for some fixed $N_0$.
Let
\[
\beta(x) = \lim_{k\to\infty} G_k(x).
\]
be the frequency of appearance of the 0 row in the symbolic expansion of $x$ (if it exists). We will also use the finite approximations
\[
P_k = G_{M\sigma^{2-k}-1}(x)
\]
and
\[
P_k'
= \sum_{j=0}^{n-1} F_{M\sigma^{2-k}}^{M\sigma^{1-k}-1}(0,j)(x),
\]
where $M>N_0$ is fixed, to be defined later.

Our first step is the standard calculation:
\begin{lem} \label{lem:calc0}
If $\beta(x) = \beta$ then

\[
d_\mu^{\rm {symb}}(x) =  \alpha(\beta) :=\frac{1}{\log m} \cdot (\sigma CH_q(\beta) + (1-\sigma) (H_q(\beta)) )).
\]
\end{lem}

Our approach is to relate the local dimension at a point $x$ to $\beta(x)$ and use the following lemma.

\begin{lem}[Gui-Li]\label{freqdim} For all $\beta\in [0,1]$ we have that
$$\dim_P\{x:\beta(x)=\beta\}=\dim_H\{x:\beta(x)=\beta\}= \frac 1 {\log m} \cdot \left(\sigma H_q(\beta) + (1-\sigma) (H_q(\beta))\right).$$
\end{lem}
\begin{proof}
For $\beta\in (0,1)$ this follows from Theorem 1.1 in \cite{GL}. For the case when $\beta=0$ or $1$ it is a simple exercise.
\end{proof} 
Under certain special cases it is straightforward to show that $\{x:\beta(x)=\beta\}=X^{\rm{symb}}_{\alpha_{\beta}}$.
\begin{lem}
\begin{enumerate}
\item
If $n_0=n_1$ and $q_0\neq q_1$ then  $\{x; \beta(x) = \beta\} = X^{\rm{symb}}_{\alpha(\beta)}$.
\item
If $q_0=q_1$ and $n_0\neq n_1$ then  $\{x; \beta(x) = \beta\} = X^{\rm{symb}}_{\alpha(\beta)}$.
\item
If $q_0=q_1$ and $n_0=n_1$ then $\alpha_m=\alpha_M$ and $X_{\alpha_m}^{\rm{symb}}=\Lambda$
\end{enumerate}
\end{lem}
\begin{proof}
For the first part fix $x\in\Lambda$ and $k\in\N$, we can calculate
$$-\log\mu(C_k(x))= kH_q(G_k(x))+\sigma_k\log n_0+\mathit{o}(k)$$
and we can see that $d_{\mu}^{\text{symb}}=\alpha(\beta)$ if and only if $\lim_{k\to\infty}G_k(x)=\beta$.
For the second part we have that if $q_0=q_1=1/2$  
$$-\log(\mu(C_k(x)))= k\log 2+k\sigma ( G_{\sigma_k}\log n_0+  (1-G_{\sigma k})\log n_1)+\mathit{o}(k).$$
Again we clearly have that $d_{\mu}^{\text{symb}}(x)=\alpha(\beta)$ if and only if $\lim_{k\to\infty}G_k(x)=\beta$. 

Finally for the 3rd part if $n_0=n+1$ then
$$-\log(\mu(C_k(x))= k\log 2+k\sigma\log n_0+\mathit{o}(k)$$
and 
$$d_{\mu}^{\text{symb}}(x)=\frac{\log 2+\sigma\log n_0}{\log m}=\dim\Lambda=\alpha_m=\alpha_M.$$
\end{proof}
Thus in the rest of the proof we assume that $n_0\neq n_1$ and $q_0\neq q_1$ and define
$$A=\frac{\log(q_0/q_1)}{\sigma\log(n_0/n_1)}.$$
In the case that $|A|\neq 1$ we will again be able to show that $x\in X_{\alpha}^{\rm{symb}}$ uniquely determines $\beta(x)$.
If $A=1$ then the set $X_\alpha^{\rm {symb}}$ contains not only points with a fixed (nonunique) $\beta(x)$ but also some additional points for which $\beta(x)$ does not exist. However, we will prove that this does not lead to an increase of either Hausdorff or packing dimension.  Finally, if $A=-1$ then the set $X_\alpha^{\rm {symb}}$ also contains some additional points with no $\beta(x)$. In this case, which is covered in part 2 of Theorem \ref{thm:2}, this leads to an increase of the packing dimension, though not the Hausdorff dimension. 

\begin{lem}
If $n_0\neq n_1$ and $q_0 \neq q_1$ then
\begin{equation} \label{eps0}
P_{k+1} =  F(P_k) + O(\varepsilon) + O(\sigma^k/M),
\end{equation}
where
\[
 F(P_k) = A^{-1} P_k + B
\]
and

\[
B=-\frac 1 {\log (q_0/q_1)} (\alpha \log m + \sigma \log p_1 +
(1-\sigma) \log q_1).
\]
\end{lem}

\begin{proof}
Given
$x\in \Lambda$ and $M\in\N$, the measure $\mu(C_{\sigma^{1-k}M}(x))$ is
precisely determined by $\{i_k(x)\}_{k=0}^{\lfloor\sigma^{1-k}M\rfloor-1}$.

We have

\begin{equation} \label{p1p1}
-\frac 1 {\sigma^{1-k}M} \log \mu(C_{\sigma^{1-k}M}(x)) = \sigma (CH_q(P_k)) + (1-\sigma) (H_q(P_k') )+ o(1).
\end{equation}

For $x\in X_{\alpha, M\sigma^{1-k}, \epsilon}$ we have

\[
\frac 1 {\sigma^{1-k}M} \log \mu(C_{\sigma^{1-k}M}(x)) \in [-(\alpha+\epsilon) \log m,
-(\alpha-\epsilon)\log m],
\]
hence \eqref{p1p1} lets us obtain a relation between $P_k$ and
$P_k'$. Applying the
obvious relation

\[
P_{k+1} = \sigma P_k + (1-\sigma) P_k' + O(\sigma^k/M)
\]
we get

\begin{equation} \label{p1p2}
\log \frac {q_0} {q_1} P_{k+1} = \log \frac {n_0} {n_1} P_k -\alpha
\log m - \sigma \log p_1 - (1-\sigma) \log q_1 + O(\epsilon) + O(\sigma^k/M)
\end{equation}
and the assertion follows.
\end{proof}


We will denote by $P$ the fixed point of $F$ and by $\mu_P$ the
Bernoulli measure given by the probability vector
$\tilde{p}_{0j}=P/n_0, \tilde{p}_{1j} = (1-P)/n_1$. We note that
$\mu_P$ is the Bernoulli measure constructed by King in \cite{K}
and that

\[
\dim_H X_\alpha = \dim_H X_\alpha^{\rm symb} = \dim_H \mu_P.
\]

The map $F$ is very simple, we only need to consider several cases
depending on the value of $A$. Note that $A$ does only depend on
$\mu$ and $\alpha$, not on $M$, $N$ or $\epsilon$. Hence, we can
make our choice of $\epsilon$ only at this moment (this will
matter in the proof of the following lemma).
\subsection{Proof of part 1 of Theorem \ref{thm:2}}
Part 1 of Theorem \ref{thm:2} corresponds to the case when $A\neq -1$. 
We will start with the simple case $|A|\neq 1$.

\begin{lem} \label{lem:an1}
If $|A| \neq 1$ then
\[
X_\alpha^{\rm {symb}} = \{x; \beta(x)=P\}.
\]
\end{lem}
\begin{proof}


It follows from Lemma \ref{lem:calc0} that $X_\alpha^{\rm {symb}}\supseteq\{x; \beta(x)=P\}$. To obtain the other direction, we will consider two cases.

Case I: $|A|>1$. In this situation we choose $\epsilon>0$ to be small relative to $|A^{-1}|$ and note that the map $F$ is contracting to the fixed point $P$.
This also means that the real frequencies $P_k$ will converge to
the region $[P-c\epsilon, P+c\epsilon]$ and then stay there. As $\epsilon$ can be chosen arbitrarily small, the assertion follows.

Case II: $|A|<1$. In this situation we choose a small $\epsilon$
and the map $P_k \to P_{k+1}$ is diverging (except in some $c\epsilon$-neighbourhood of $P$, where the error term can dominate the divergence of $F$). If $P_k$ does not belong to 
$[P-c\epsilon, P+c\epsilon]$ then $P_{k+1}$ will be even further from
$P$ and so on. However, for any point $x$ all the frequencies $P_k$
must belong to $[0,1]$, hence they must indeed be in some
$c\epsilon$-neighborhood of $P$ for all $k>0$. Like in the first
case, $\epsilon$ can be chosen arbitrarily small and the assertion follows.
\end{proof}
We now have that by Lemma \ref{freqdim} and Lemma \ref{lem:eqsymb}
$$\dim_H X_{\alpha}\leq\dim_P X_{\alpha}\leq \dim_P X_{\alpha}^{\text{symb}}=\dim_H X_{\alpha}^{\text{symb}}.$$
The assertion of the Theorem when $|A|\neq 1$ now follows since by \cite{JR} we know $\dim_H X_{\alpha}=\dim_H X_{\alpha}^{\text{symb}}.$

We are left with the case when $A= 1$. In this case and in the case in the following subsection, when $A=-1$, the symbolic local dimension of $x$ does not determine $\beta(x)$. However when $A=1$ we have the following statement:

\begin{lem} \label{lem:a1}
If $A=1$ then $\alpha_m=\alpha_M$ and
\begin{equation}\label{degenerate}
\dim_P X_{\alpha_m}^{\rm{symb}}=\dim_H X_{\alpha_m}^{\rm{symb}}=\dim_H \Lambda.
\end{equation}
\end{lem}
\begin{proof}
It is clear that $\alpha_m=\alpha_M$ and that $\mu$ is the measure of maximal dimension. We let $\alpha=\alpha_m$.
 We have for all $x\in
X_{\alpha, N, \epsilon}$
\[
|P_{k+1}-P_k|\leq O(\epsilon)+ O(\sigma^k/M).
\]

Hence, a drift is possible: for any $\tilde{M}>N$ we might not
know the frequency $Q=\sum_j F_0^{\tilde{M}\sigma^{-1}-1}(0,j)(x)$.
Still, the set $X_{\alpha, N, \epsilon}$ can be covered by

\[
\sup_{Q\in [0,1]} \exp(\tilde{M} CH(Q)) \exp((\sigma^{-1}-1)\tilde{M}H(Q))  \exp(\tilde{M} O(\epsilon))
\]
approximate squares of level $\tilde{M} \sigma^{-1}$ and hence

\[
\overline{\dim_B} X_{\alpha, N, \epsilon} \leq \sup_{Q\in [0,1]}
\frac 1 {\log m} \left(\sigma CH(Q)+(1-\sigma)H(Q)\right)
+O(\epsilon) = \dim_H \Lambda + O(\epsilon).
\]

As $\dim_H X_\alpha^{\rm symb} = \dim_H \Lambda$, \eqref{degenerate}
follows.
\end{proof}

The proof of part 1 of Theorem \ref{thm:2} can now be completed since by Lemma \ref{lem:eqsymb} we have
$$\dim_H X_{\alpha_m}\leq \dim_P X_{\alpha_m}=\dim_H X_{\alpha_m}=\dim_H\Lambda$$
and we know $\dim_H X_{\alpha_m}=\dim_H\Lambda$.

\subsection{Proof of part \ref{thm2p3} of Theorem \ref{thm:2}}

Finally, the most interesting case when $A=-1$ which corresponds to part \ref{thm2p3} of Theorem \ref{thm:2}. Our goal is to prove the following theorem from which part \ref{thm2p3} of Theorem \ref{thm:2} follows.
\begin{thm}\label{A=-1}
Assume $A=-1$. If $\alpha \in (\alpha_m, \alpha_M)$ then
\[
\dim_H X_\alpha^{\rm {symb}} < \dim_P X_\alpha^{\rm {symb}}=\dim_P X_\alpha.
\]
If $\alpha \in \{\alpha_m, \alpha_M\}$ then
\[
\dim_H X_\alpha^{\rm {symb}} = \dim_P X_\alpha^{\rm {symb}}\leq \dim_H X_{\alpha}.
\]
\end{thm}
The proof of this theorem will be split into several smaller parts. The map
$F$ has only one fixed point $P$, but $F^2 \equiv \text{id}$. Like
in the case when $A=1$,  drift is possible so we do not know the frequencies
$Q_1=G_{\tilde{M}}(x)$ or $Q_2=
G_{\tilde{M}\sigma^{-1}}(x)$, but we know that $Q_1+Q_2=2P
+ O(\epsilon)$. Once again one can calculate

\begin{equation} \label{delta}
\overline{\dim_B} X_{\alpha, N, \epsilon} \leq \sup_{\rho} \frac 1
{\log m} \left(\sigma CH(P+\rho) + (1-\sigma) H(P-\rho)\right)+O(\epsilon),
\end{equation}
where $\delta = \rho (1+\sigma)/(1-\sigma)$ and the supremum is
taken over $\rho$ such that $P\pm \delta \in [0,1]$.

The first thing to note is that if $P\in\{0,1\}$ (which corresponds to local dimensions 
$\alpha_{\rm min}$ and $\alpha_{\rm max}$) then 
$\rho=0$ is the only admissible choice. In these cases we have equality in \eqref{delta} and
we have the following simple result.
\begin{lem}\label{endpoints}
For $A=-1$ and $\alpha\in\{\alpha_m,\alpha_M\}$ we have that
$$\dimp X_{\alpha}^{\text{symb}}=\dim_H X_{\alpha}^{\text{symb}}\leq \dim_H X_{\alpha}.$$
\end{lem}
\begin{proof}
Let $\alpha\in\{\alpha_m,\alpha_M\}$ and thus $P\in\{0,1\}$.  Thus in inequality (\ref{delta}) the only choice of $\rho=0$ and so
$$\overline{\dim_B} X_{\alpha, N, \epsilon} \leq \sup_\rho \frac {\sigma CH(P)}{\log m} +O(\epsilon)$$
which means that
$$\dimp X_{\alpha^{\text{symb}}}\leq \frac {\sigma CH(P)}{\log m}.$$
To complete the proof we need to show that $\dim_H X_{\alpha^{\text{symb}}}\geq \frac {\sigma CH(P)}{\log m}$ and $\dim_H X_{\alpha}\geq \frac {\sigma CH(P)}{\log m}$.
To do this we observe that either $\Lambda\cap\{(0,y):y\in\R\}\subseteq X_{\alpha^{\text{symb}}}\cup X_{\alpha}$ or $\Lambda\cap\{(1,y):y\in\R\}\subseteq X_{\alpha^{\text{symb}}}\cup X_{\alpha}$. We can then easily calculate
$$\dim_H \Lambda\cap\{(0,y):y\in\R\}=\frac{\log n_0}{\log n}=\frac {\sigma CH(1)}{\log m}$$
and
$$\dim_H \Lambda\cap\{(1,y):y\in\R\}=\frac{\log n_1}{\log n}=\frac {\sigma CH(0)}{\log m}.$$
\end{proof}

However, in the interior of the spectrum the packing and Hausdorff
symbolic spectra are different.

\begin{defn}
Given $\delta, K_0>0$ let $W_{\alpha, \delta, K_0}$ be the set of points with following properties:
\begin{itemize}
\item for $K$ even for any $a,b \in [\sigma^{-K}, \sigma^{-K+1}-1]$ we have

\[
\left|  \sum_{j=0}^{n-1} F_a^{b-1}(0,j)(x) - P-\delta \right| < \frac {K_0} {b-a},
\]
\item for $K$ odd for any $a,b \in [\sigma^{-K}, \sigma^{-K+1}-1]$ we have

\[
\left| \sum_{j=0}^{n-1} F_a^{b-1}(0,j)(x) - P+\delta \right| < \frac {K_0} {b-a}.
\]
\end{itemize}
\end{defn}

If $x\in W_{\alpha, \delta, K_0}$ then whenever for some large $M$
\[
G_{M-1}(x) = P+\delta',
\]
there exists $\delta'' = \delta' + O(1/M)$ such that for all $K\in \N$

\[
G_{M\sigma^{-K}-1}(x) = P+(-1)^K \delta''+O(\sigma^K/M).
\]

Obviously, $|\delta''|\leq \delta/(1+\sigma)$. It follows that each of the
sets $W_{\alpha, \delta, K_0}$ is contained in $X_\alpha^{\rm
symb}$.

Our nearest goal is to calculate the upper box counting dimension of $W_{\alpha, \delta, K_0}$. We will need the following simple lemma.

\begin{lem} \label{lem:calc}
Let $M$ be large. Let $W\subseteq \{0,1\}^M$ be the set of words $\omega$ for which for all subwords $(\omega_a,\ldots, \omega_{b-1})$

\[
\left| F_a^{b-1}(0) - P\right| < \frac {K_0} {b-a}.
\]
Then
\[
\frac 1 M \log |W| \geq H(P) - O\left(\frac {\log K_0} {K_0} \right) - O\left(\frac {K_0} {M} \right).
\]

\end{lem}
\begin{proof}
Consider the set of sequences for which for all $0\leq i<2M/K_0$
\[
\left| F_{iK_0/2}^{(i+1)K_0/2 - 1}(0) - PK_0/2 \right| \leq 1.
\]

Moreover, if $PK_0/2$ is not an integer, we choose the blocks for which
\[
F_{iK_0/2}^{(i+1)K_0/2 - 1}(0) - PK_0/2 <0
\]
and blocks for which
\[
F_{iK_0/2}^{(i+1)K_0/2 - 1}(0) - PK_0/2 >0
\]
in such a way that for any $i,j$

\[
\left| F_{iK_0/2}^{jK_0/2 - 1}(0) - PK_0/2 \right| \leq 1
\]
(the blocks for which the frequency was greater than the target and blocks for which it was smaller than the target form a Rauzy sequence).

Let $\tilde{W} \subseteq \{0,1\}^M$ be the set we defined.
Clearly, $\tilde{W} \subseteq W$. We can estimate by Stirling's formula

\[
\log |\tilde{W}| \geq \left\lfloor \frac{2M}{K_0} \right\rfloor \cdot \log {\binom{K_0/2}{PK_0/2}} = 2M H(P) + O\left(\frac{M\log K_0}{K_0}\right) + O(K_0).
\]
\end{proof}


Let  $Y(\delta)=\sup_{0 \leq \gamma \leq 2} \tilde{Y}(\gamma,\delta)$, where

\begin{eqnarray} \label{eqn:ytilde}
\tilde{Y}(\gamma, \delta) &=& H(P) + \sigma (P\log n_0 + (1-P) \log n_1) \nonumber\\
&+& \left(1-\sigma^{\gamma -\lfloor \gamma\rfloor}\right)\Delta((-1)^{\lfloor \gamma \rfloor}) \nonumber \\
&+& \frac 1 {1+\sigma} \left(\sigma^{-2\lfloor \gamma/2 \rfloor} \Delta(-1) + \sigma^{-2\lfloor (\gamma-1)/2 \rfloor +1} \Delta(+1)\right)\nonumber\\
&+& (-1)^{\lfloor \gamma \rfloor} \left( \frac {2\sigma^{\gamma-\lfloor \gamma \rfloor}} {1+\sigma} -1\right) \delta \log \frac {n_0} {n_1}
\end{eqnarray}
and
\[
\Delta(s) := H(P+s\delta) - H(P).
\]

Note that the first line of \eqref{eqn:ytilde} has an important geometric meaning:
\[
H(P) + \sigma (P\log n_0 + (1-P) \log n_1) = \log m \cdot \dim_H X_\alpha.
\]

\begin{lem} \label{lem:boxdim}
\[
\lim_{K_0 \to \infty} \overline{\dim}_B W_{\alpha, \delta, K_0} = \frac {Y(\delta)} {\log m}.
\]
\end{lem}
\begin{proof}
We will estimate the number $Z_r$ of
approximate squares of level $\sigma^{-r}$ necessary to cover
$W_{\alpha, \delta, K_0}$. Here $r$ is not necessarily an integer.

Let
\[
Q(r) :=\sum_{j=0}^{n-1} F_0^{\sigma^{-r+1}-1}(0,j)
\]

We can calculate
\[
Q(r)=  (\sigma^{1-r} - \sigma^{1-\lfloor r \rfloor}) (P+(-1)^{\lfloor r \rfloor-1}\delta) + \sum_{\ell =1}^{\lfloor r \rfloor -1} (\sigma^{-\ell} - \sigma^{1-\ell})(P+(-1)^{1-\ell}\delta) + O(K_0 r).
\]


The geometric series is easy to sum, we get

\begin{equation} \label{eqn:qr}
\sigma^{r-1} Q(r) = P + (-1)^{\lfloor r \rfloor} \left(\frac {2\sigma^{r-\lfloor r \rfloor}} {1+\sigma} - 1\right) \delta + O(\sigma^r K_0).
\end{equation}

We can now write a formula for $Z_r$, using Lemma \ref{lem:calc}:
\begin{align} 
\log Z_r &=& (\sigma^{-r}-\sigma^{-\lfloor r \rfloor}) H(P+(-1)^{\lfloor r \rfloor} \delta) + \sum_{\ell =1}^{\lfloor r \rfloor} (\sigma^{-\ell} - \sigma^{1-\ell}) H(P+(-1)^{1-\ell}\delta)\nonumber\\
 &+& Q(r) \log n_0 + (\sigma^{1-r} - Q(r))\log n_1 + O\left(\frac {r\log K_0} {K_0} \right) + O\left(K_0\right).
\label{eqn:zr}
\end{align}

We can sum the geometric series and substitute \eqref{eqn:qr}, obtaining the simple formula

\begin{eqnarray*}
\log Z_r &=& \sigma^{-r} H(P) + \sigma^{1-r} (P\log n_0 + (1-P) \log n_1)\\
 &+& \left(\sigma^{-r} - \sigma^{-\lfloor r \rfloor}\right) \Delta((-1)^{\lfloor r \rfloor})\\
 &+& \frac 1 {1+\sigma} \left(\sigma^{-2\lfloor r/2 \rfloor} \Delta(-1) + \sigma^{-2\lfloor (r-1)/2 \rfloor +1} \Delta(+1)\right)\\
  &+& (-1)^{\lfloor r \rfloor} \left( \frac {2\sigma^{r-\lfloor r \rfloor}} {1+\sigma} -1\right) \delta \log \frac {n_0} {n_1}+ O\left(\frac {r\log K_0} {K_0} \right) + O(K_0).
\end{eqnarray*}

We want to calculate the upper limit

\[
\overline{\dim_B} W_{\alpha,
  \delta, K_0} = \limsup \sigma^r \frac {\log Z_r} {\log m}.
\]
Note that over any subsequence $r=\gamma, \gamma+2, \gamma+4,\ldots$ the
sequence $\lim_{k\to\infty}   \sigma^{\gamma+2k}\log Z_{\gamma+2k}$ converges:

\begin{eqnarray} \label{eqn:gamma}
\lim_{k\to\infty}   \sigma^{\gamma+2k}\log Z_{\gamma+2k} &=& H(P) + \sigma (P\log n_0 + (1-P) \log n_1)\nonumber\\
&+& \left(1-\sigma^{\gamma-\lfloor \gamma\rfloor}\right) \Delta((-1)^{\lfloor \gamma \rfloor})\nonumber\\
&+& \frac 1 {1+\sigma} \left(\sigma^{-2\lfloor \gamma/2 \rfloor} \Delta(-1) + \sigma^{-2\lfloor (\gamma-1)/2 \rfloor +1} \Delta(+1)\right)\nonumber\\
&+& (-1)^{\lfloor \gamma \rfloor} \left( \frac {2\sigma^{\gamma-\lfloor \gamma \rfloor}} {1+\sigma} -1\right) \delta \log \frac {n_0} {n_1}
\end{eqnarray}

As the right hand side of \eqref{eqn:gamma} is exactly the function we denoted as $\tilde{Y}(\gamma, \delta)$, the assertion follows.


\end{proof}

We do not need the exact formulation of the Lemma \ref{lem:boxdim}. We only need the following corollary:

\begin{cor} \label{cor:form}
For $\alpha \notin \{\alpha_m, \alpha_M\}$ and sufficiently small $\delta > 0$,
\[
\frac 1 {\log m} Y(\delta) > \dim_H X_\alpha.
\]
\end{cor}
\begin{proof}

Define
\[
\tilde{Z}(\gamma, \delta) = \tilde{Y}(\gamma, \delta) - \log m \cdot \dim_H X_\alpha.
\]

We can write the approximate form for $\partial/\partial \delta \tilde{Z}(\gamma, \delta)$:

\begin{equation}
\frac \partial {\partial \delta} \tilde{Z}(\gamma, \delta) = O(\delta)+
\end{equation}
\[
(-1)^{\lfloor \gamma \rfloor} \left(H'(P) - \log \frac {n_0} {n_1} + \sigma^{\gamma - \lfloor \gamma \rfloor} (-H'(P) + \frac 2 {1+\sigma} \log \frac {n_0} {n_1}) + \sigma^{\lfloor \gamma \rfloor - \gamma} \frac {1-\sigma} {1+\sigma} H'(P) \right).
\]

As $\tilde{Z}(\gamma,0) \equiv 0$ and

\[
\frac \partial {\partial \delta} \tilde{Z}(\gamma, \delta) + \frac \partial {\partial \delta} \tilde{Z}(\gamma +1, \delta) = O(\delta),
\]
the only possibility for $\sup_\gamma \tilde{Z}(\gamma, \delta)$ to stay nonpositive for small $\delta$ is that the function
\[
H'(P) - \log \frac {n_0} {n_1} + \sigma^\gamma \left(-H'(P) + \frac 2 {1+\sigma} \log \frac {n_0} {n_1}\right) + \sigma^{- \gamma} \frac {1-\sigma} {1+\sigma} H'(P)
\]
is equal to 0 for all $\gamma \in [0,1]$. However, this function is a linear combination of functions $1$, $\sigma^\gamma$, and $\sigma^{-\gamma}$, which are linearly independent. Hence, all the coefficients must be equal to zero. In particular, $H'(P)=0$ and $H'(P) - \log (n_0/n_1) = 0$. However, this implies $n_0 = n_1$, which is a contradiction.
\end{proof}
The last part in proving Theorem \ref{A=-1} is to combine the following proposition with Lemma \ref{lem:eqsymb}.
\begin{prop} \label{prop:dimp}
\[
\dim_P X_\alpha^{\rm symb} \leq \max_{\delta \leq \min(P, 1-P)} \frac {Y(\delta)} {\log m}
\]
\[
\dim_P X_\alpha \geq \max_{\delta \leq \min(P, 1-P)} \frac {Y(\delta)} {\log m}
\]
\end{prop}
\begin{proof}
Let us start with the lower bound.

Consider any $\delta < \min(P, 1-P)$.
Let $\gamma_0(\delta)$ be such that
\[
Y(\delta) = \tilde{Y}(\gamma_0, \delta).
\]

Consider the measure $\nu_\delta$ on $D^\N$ defined as follows. For any even $K$, for all $\ell \in [\sigma^{-K}, \sigma^{-K-1})$ we choose $i_\ell=0$ with probability $P+\delta$ and $i_\ell=1$ with probability $1-P-\delta_0$, independently. For any odd $K$, for all $\ell \in [\sigma^{-K}, \sigma^{-K-1})$ we choose $i_\ell=0$ with probability $P-\delta_0$ and $i_\ell=1$ with probability $1-P+\delta_0$. Whichever the choice of $i_\ell$, all the possible $j_\ell; (i_\ell, j_\ell)\in D$ we choose with the same probability $1/n_{i_\ell}$.

We will use also the projection of $\nu_\delta$ onto $\Lambda$, which we will also denote by $\nu_\delta$.

Let us begin from the observation, that for every $\epsilon>0$ and for $\nu_\delta$-almost every $x$ there are only finitely many $N$ such that all $i_\ell; \ell = N, N+1,\ldots, N(1+\epsilon)$ are equal. Hence,

\begin{equation} \label{eqn:uplocdim}
\overline{d}_{\nu_\delta}(x) = \limsup_{\ell\to\infty} \frac {\log \nu_\delta(C_\ell(x))} {-\ell \log m} = \overline{\delta}_{\nu_\delta}(x)
\end{equation}
$\nu_\delta$-almost everywhere.

To obtain the lower bound, we need the following two lemmas.

\begin{lem}
\[
\nu_\delta(X_\alpha) =1.
\]
\end{lem}
\begin{proof}

By \eqref{eqn:uplocdim}, $\nu_\delta$-typical point is in $X_\alpha$ if and only if it is in $X_\alpha^{\rm{symb}}$.
Recall that a sufficient condition for $x\in X_\alpha^{\rm symb}$ is that there exists a function $\epsilon_N\to 0$ such that for all except finitely many $N$ we have

\begin{equation} \label{eqn:ldev}
\left| \frac 1 {\sigma N} \sum_{j=0}^{n-1} F_0^{\sigma N-1}(0,j)(x) + \frac 1 N \sum_{j=0}^{n-1} F_0^{N-1}(0,j)(x) -2P \right| < \epsilon_N.
\end{equation}

Standard large deviation estimates show that for $\nu_\delta$-typical $x$ \eqref{eqn:ldev} is satisfied for all except finitely many $N$ if we choose $\epsilon_N = N^{-1/3}$. We skip the details.
\end{proof}

\begin{lem}
For $\nu_\delta$-almost every $x$
\[
\overline{d}_{\nu_\delta}(x) \geq \frac 1 {\log m} Y(\delta).
\]
\end{lem}
\begin{proof}

To prove the assertion, we just need to find a sequence $\ell\to\infty$ for which the limit in \eqref{eqn:uplocdim} would be not smaller than $Y(\delta)/\log m$. The right sequence is
\[
\ell_K = \sigma^{a_K},
\]
where
\[
a_K = 2K + \gamma_0(\delta).
\]
For $\nu_\delta$-typical $x$, for $K$ large enough we have

\begin{eqnarray*}
\log \nu_\delta (C_{\ell_K}(x)) &=& -(\sigma^{-a_K}-\sigma^{-\lfloor a_K \rfloor}) H(P+(-1)^{\lfloor a_K \rfloor} \delta)\\
&-&\sum_{k =1}^{\lfloor a_K \rfloor} (\sigma^{-k} - \sigma^{1-k}) H(P+(-1)^{1-k}\delta)\\
&-& \left((\sigma^{1-a_K} - \sigma^{1-\lfloor a_K \rfloor}) (P+(-1)^{\lfloor a_K \rfloor-1}\delta)\right)\log n_0 \\
&-& \left(\sum_{k =1}^{\lfloor a_K \rfloor -1} (\sigma^{-k} - \sigma^{1-k})(P+(-1)^{1-k}\delta)\right) \log n_0 \\
&-& \left(\sigma^{1-a_K} - (\sigma^{1-a_K} - \sigma^{1-\lfloor a_K \rfloor}) (P+(-1)^{\lfloor a_K \rfloor-1}\delta)\right)\log n_1\\ 
&+& \left(\sum_{k =1}^{\lfloor a_K \rfloor -1} (\sigma^{-k} - \sigma^{1-k})(P+(-1)^{1-\ell}\delta)\right) \log n_1 + O(\ell_K^{1/2}).
\end{eqnarray*}

Comparing with \eqref{eqn:zr}, we get
\[
\log \nu_\delta (C_{\ell_K}(x)) = -Z_{\ell_K}+\mathit{o}(\ell_K).
\]
and
\[
\lim_{K\to\infty} \frac 1 {\ell_K} \log \nu_\delta (C_{\ell_K}(x)) = - \lim_{K_0\to\infty} \lim_{K\to\infty} \frac 1 {\ell_K} Z_{\ell_K}.
\]

The calculations of this limit were done in the course of proof of Lemma \ref{lem:boxdim}:

\[
\lim_{K_0\to\infty} \lim_{K\to\infty} \frac 1 {\ell_K} Z_{\ell_K} = Y(\delta).
\]
The assertion follows.
\end{proof}

To finish the proof of the lower bound in Proposition \ref{prop:dimp} we need only to observe that $Y$ is a continuous function of $\delta$, hence
\[
\max \{Y(\delta), \delta \leq \min(P, 1-P)\} = \sup \{Y(\delta), \delta < \min(P, 1-P)\},
\]
and apply Frostman's lemma.

To prove the upper bound, let us fix some small $\epsilon_1$ and some much smaller $\epsilon_2<\epsilon_1/2U$. Let us consider a finite family of intervals $\{I_k\}_{k=1}^V$ of size $\epsilon_1$, covering $[0,1]$. Let $J_k=3I_k$, that is, an interval with the same center as $I_k$ but three times longer.

For every point $x\in X_\alpha^{\rm symb}$ there is some minimal $N(x)$ such that for all $N>N(x)$
\begin{equation} \label{eqn:nx}
\left| G_{\sigma N-1}(x) + G_{N-1}(x) -2P \right| < \epsilon_2.
\end{equation}

We divide $X_\alpha^{\rm symb}$ into subsets
\[
X_{\alpha, N} = \{x\in X_\alpha^{\rm symb}; N(x) = N\}.
\]

Our goal is to prove
\begin{lem}
For each $N$,
\[
\overline{\dim}_B X_{\alpha, N} \leq \frac 1 {\log m} \max_{\delta \leq \min(P, 1-P)} Y(\delta).
\]
\end{lem}
\begin{proof}
Let $x\in X_{N, \alpha}$. Let $k$ be such that
\[
G_{N-1}(x) \in I_k.
\]
By \eqref{eqn:nx}, it means that for all $\ell \leq U$ we have
\[
G_{\sigma^{-2\ell}N-1}(x) \in J_k\text{ and }
G_{\sigma^{-2\ell+1}N-1}(x) \in 2P-J_k
\]
We can estimate the number $A(k, N)$ of possible sequences $(i_N(x), \ldots, i_{\lfloor\sigma^{-2U} N\rfloor})$. Like in the lower bound, the estimation will be almost the same as in the proof of Lemma \ref{lem:boxdim}:

\[
\frac 1 {(\sigma^{-2U}-1)N}\log A(k,N) \leq \sup_{\delta; \{P+ \delta, P-\delta\} \cap J_k \neq \emptyset} Y(\delta) + O(\epsilon_1).
\]

There are only $V$ possible $k$'s, hence the number $B(N)$ of possible sequences $(i_N(x), \ldots, i_{\lfloor\sigma^{-2U} N\rfloor})$ for all $x\in X_{\alpha, N}$ satisfies

\[
\frac 1 {(\sigma^{-2U}-1)N} \log B(N) \leq \sup_{\delta \leq \min(P, 1-P)} Y(\delta) + O(\epsilon_1) + O\left( \frac 1 {\sigma^{-2U}N}\right).
\]

Repeating the argument for $B(\sigma^{-2U}N)$ and so on and passing to the limit, we get
\[
\overline{\dim}_B X_{\alpha, N} \leq \frac 1 {\log m} \min(P, 1-P) Y(\delta) + O(\epsilon_1)
\]
and the assertion follows.
\end{proof}

As the packing dimension is not greater than the upper box counting dimension, this gives the upper bound for Proposition \ref{prop:dimp} and so the proof of Proposition \ref{prop:dimp} is complete. 
\end{proof}
Theorem \ref{A=-1} now follows by combining Proposition \ref{prop:dimp}, Lemma \ref{lem:eqsymb} Lemma \ref{endpoints} and Corollary \ref{cor:form}. This completes the proof of Theorem \ref{thm:2}.


\begin{thebibliography}{WW}


\bibitem[AP]{AP}
M. Arbeiter and N. Patzschke, 
\newblock Random self-similar multifractals, 
\newblock {\it Math. Nachr.} 181 (1996), 5–-42.

\bibitem[BOS]{BOS}
I. S. Baek, L. Olsen, N. Snigireva
\newblock Divergence points of self-similar measures and packing dimension,
\newblock {\it Adv. Math.} 214 (2007), no. 1, 267–-287. 

\bibitem[BF]{BF}
J. Barral and D. J. Feng,
\newblock Weighted thermodynamic formalism on subshifts and applications,
\newblock {\it Asian J. Math.} 16 (2012), no. 2, 319–-352.

\bibitem[BM]{BM}
J. Barral and M. Mensi, 
\newblock Gibbs measures on self-affine Sierpi\'nski carpets and
their singularity spectrum, 
\newblock {\it Ergodic Theory Dynam. Systems} 27 (2007), no. 5, 1419-–1443.

\bibitem[B]{B}
T. Bedford, 
\newblock {\it Crinkly curves, Markov partitions and box dimensions in self-similar
sets,} 
\newblock PhD Thesis, The University of Warwick, 1984.

\bibitem[CM]{CM}
R. Cawley and R.D. Mauldin, 
\newblock Multifractal decompositions of Moran fractals,
\newblock {\it Adv. Math.} 92 (1992), no. 2, 196-–236.

\bibitem[GL]{GL}
Y. Gui and W. Li,
\newblock Hausdorff dimension of fiber-coding sub-Sierpi\'{n}ski carpets,
\newblock {\it J. Math. Anal. Appl.} 331 (2007), 62--68.

\bibitem[JR]{JR}
T. Jordan and M. Rams,
\newblock Multifractal analysis for Bedford-McMullen carpets,
\newblock {\it Math. Proc. of Camb. Phil. Soc.} 150 (2011), 147--156.

\bibitem[K]{K}
J. King,  
\newblock The singularity spectrum for general Sierpi\'nski carpets,
\newblock {\it Adv. Math.} 116 (1995), 1--8.

\bibitem[Ma]{Mat}
P. Mattila
\newblock
Geometry of sets and measures in Euclidean spaces. Fractals and rectifiability
\newblock
 Cambridge Studies in Advanced Mathematics, 44. Cambridge University Press, Cambridge, 1995.
\bibitem[McM]{M}
C. McMullen, 
\newblock The Hausdorff dimension of general Sierpi\'nski carpets, 
\newblock {\it Nagoya Math. J.} 96 (1984), 1-–9.

\bibitem[N]{N} 
O. A. Nielsen,
\newblock  The Hausdorff and packing dimensions of some sets related to Sierpi\'nski carpets, 
\newblock {\it Can. J. Math.} 51 (1999), 1073–-88 .

\bibitem[O]{O}
L. Olsen, 
\newblock Self-affine multifractal Sierpi´nski sponges in $\R^d$, 
\newblock {\it Pacific J. Math.} 183 (1998), no. 1, 143–-199.

\bibitem[R]{R}
H. W. J. Reeve,
\newblock  The packing spectrum for Birkhoff averages on a self-affine repeller,
\newblock {\it Ergodic Theory Dynam. Systems} 32 (2012), no. 4, 1444–-1470. 




\end{thebibliography}
\end{document}